\documentclass[11pt]{amsart}
\usepackage{amssymb}
\usepackage{amsrefs}
\usepackage[text={390pt,660pt},centering]{geometry}

\newcommand{\norm}[1]{\lVert#1\rVert}
\newcommand{\inner}[1]{\langle#1\rangle}
\newcommand{\N}{\mathbb{N}}
\newcommand{\R}{\mathbb{R}}
\newcommand{\kk}{_{\smash{k_j}}}

\let\nb\nobreak

\newtheorem*{Farkas}{Farkas' lemma}
\newtheorem*{lemma}{Lemma}

\theoremstyle{definition}
\newtheorem{remark}{Remark}

\title[Closedness of Convex Cones and Farkas' Lemma]
	{A Short Simple Proof of Closedness of Convex Cones and Farkas' Lemma}
\author{Wouter Kager}
\address{Vrije Universiteit, Department of Mathematics, Amsterdam, the 
Netherlands}

\begin{document}

\maketitle

Farkas' lemma is a cornerstone of the theory of linear inequalities, and is 
often the starting point for proving the duality theorem of linear programming 
(see, e.g., \cite{AK2004} and~\cite{MG2007}*{Ch.~6}). Suppose $A$ is a real 
$m\times\nb n$ matrix with column vectors $a_1, \dots, a_n$, and $b$ is a 
vector in~$\R^m$. To avoid trivial cases, assume $A \neq\nb 0$ and (with no 
loss of generality) $a_1 \neq\nb 0$. Let~$K$ be the convex cone $\{ Ax\colon x 
\in\nb \R_+^n \}$, where $\R_+$ denotes the set of non-negative reals. Then

\begin{Farkas}
	Either $b$ lies in~$K$, or there exists a vector~$y$ in~$\R^m$ such that 
	$\inner{a_1,y} \geq\nb 0, \dots, \inner{a_n,y} \geq\nb 0$ and $\inner{b,y} 
	<\nb 0$ (but not both).
\end{Farkas}

In~\cite{Ko1998}, Komornik gave a simple proof of Farkas' lemma, and stated as 
a difficulty with ``intuitive'' proofs that \textit{``the proof of the 
closedness of~$K$ is not obvious''}. Broyden in fact calls this \emph{``the 
most difficult part of a geometric proof''}~\cite{Br1998}. But here we show 
that closedness of~$K$ follows simply from the Bolzano--Weierstrass theorem 
once we have introduced the notion of an \emph{optimal} vector in~$[0,1]^n$. 
Farkas' lemma follows from this by standard means.

To illustrate why it is not obvious that~$K$ is closed, let $C$ be the set of 
points~$(x,y)$ such that $-1 <\nb x <\nb 1$ and $y \leq \ln(1-x^2)$. Then $C$ 
is convex and closed in~$\R^2$, but the convex cone generated by~$C$, i.e., 
the set~$\{\lambda z\colon \lambda \in\nb \R_+, z \nb\in C\}$, is the open 
lower half-plane in~$\R^2$ plus the point~0, which is not closed. Also, the 
linear map $f\colon (x,y) \mapsto\nb x$ maps~$C$ to the open interval~$(-1, 
1)$. So it is not true that a set is closed simply because it is the convex 
cone generated by a closed set of vectors, or because it is the linear image 
of such a set. We need a more refined argument.

We call a vector~$u = (u^1,\dots,u^n)$ \emph{optimal} if $u$ lies 
in~$[0,1]^n$, has length~$\norm{u} =\nb 1$, and there is no vector~$z$ 
in~$\R_+^n$ which satisfies $Az =\nb Au$ and has fewer non-zero components 
than~$u$. Note that $Au =\nb 0$ is impossible if $u$ is optimal, but any 
vector $v =\nb Ax$ with $x \in\nb \R_+^n$ can be written in the form~$\lambda 
\, Au$ with $\lambda \in\nb \R_+$ and~$u$ optimal: we can take $\lambda =\nb 
0$, $u =\nb (1,0,\dots,0)$ if $v=\nb0$, and if $v\neq\nb 0$ we can first 
choose from the set~$\{z\in\nb \R_+^n \colon v=\nb Az\}$ a vector~$z$ with a 
minimal number of non-zero components, and then take $\lambda =\nb \norm{z}$, 
$u=\nb \norm{z}^{-1}z$. We use this to prove closedness of~$K$.\looseness=-1

\begin{lemma}
	Let $K$ be the convex cone defined above, and suppose $\{v_k\}$ is a 
	sequence of points in~$K$ that converges to some limit~$v$. Then $v$ lies 
	in~$K$.
\end{lemma}

\begin{proof}
	Write $v_k = \lambda_k \, Au_k$ with $\lambda_k \in\nb \R_+$ and $u_k = 
	(u_k^1,\dots,u_k^n)$ an optimal vector. Then the sequence~$\{u_k\}$ has a 
	subsequence~$\{u\kk\}$ that converges to some limit~$u$ in~$[0,1]^n$ of 
	length~1. Let $I$ be the set of indices~$i$ for which $u^i >\nb 0$, and 
	choose $k$ in~$\{k_j\}$ so that $u_k^i >\nb 0$ for each~$i$ in~$I$. Define 
	$\mu := \min_{i\in I} \, (\, u_k^i /\nb u^i \,)$ and $z := u_k -\nb \mu 
	u$. Then $\min_{i\in I} \, (\, z^i /\nb u^i \,) =\nb 0$ and $z^i =\nb 
	u_k^i$ if $i\notin\nb I$, so $z$ lies in~$\R_+^n$ and has at least one 
	non-zero component less than~$u_k$. Since $u_k$ was optimal, this implies 
	$Az \neq\nb Au_k$, hence $Au \neq\nb 0$. But we know that $\norm{Au\kk} 
	\to\nb \norm{Au}$ and $\lambda\kk \, \norm{Au\kk} = \norm{v\kk} \to\nb 
	\norm{v}$. It follows that $\lambda\kk \to\nb \lambda := \norm{Au}^{-1} 
	\norm{v}$, hence $v =\nb \lambda \, Au$, which lies in~$K$.
\end{proof}

We do not claim complete originality of the ideas used in the proof, but are 
not aware of other publications that prove the lemma in exactly the same way. 
In particular, while this note was under review it was brought to the author's 
notice that Bonnans and Shapiro~\cite{BS2000}* {Prop.~2.41} use optimal 
vectors in a similar way, yet their proof is more involved than necessary and 
not as direct as the one presented here. Other closely related but distinct 
proofs can be found for instance in~\cites{Bl2011, daw2017} and~\cite{MG2007}* 
{pp.~96--97}. We close with a proof of Farkas' lemma:

\begin{proof}[Proof of Farkas' lemma]
	Suppose $b$ does not lie in~$K$. Set $\delta := \inf_{v\in K} \, \norm{v 
	-\nb b}$ and for each~$k$ in~$\N$, choose a point~$v_k$ in~$K$ such that 
	$\norm{v_k -\nb b} < \delta +\nb k^{-1}$. Then $\norm{v_k} \leq \norm{b} 
	+\nb \delta +\nb 1$, so the sequence~$\{v_k\}$ is bounded and hence has a 
	subsequence~$\{v\kk\}$ which converges to some~$v$. By the lemma, $v 
	\in\nb K$, hence $\delta = \norm{v -\nb b} >\nb 0$. Now write $y = v -\nb 
	b$, and let~$w$ be one of the vectors $a_1, \dots, a_n$ or~$-v$. Note that 
	then $\norm{y} =\nb \delta$, and $\norm{y +\nb \lambda w} \geq\nb \delta$ 
	for~$\lambda$ in~$(0,1)$ because $v +\nb \lambda w \in K$. 
	Therefore,\looseness=-1
	\[
		0
		\leq \lim_{\lambda\to 0^+} \frac{1}{2\lambda}
				\bigl(\, \norm{y+\lambda w}^2 - \norm{y}^2 \,\bigr)
		=	 \lim_{\lambda\to 0^+} \frac\lambda2 \, \norm{w}^2 + \inner{w,y}
		=	 \inner{w,y}.
	\]
	This gives $\inner{a_1,y} \geq\nb 0, \dots, \inner{a_n,y} \geq\nb 0$ and 
	$-\inner{b,y} = \inner{-v,y} +\nb \norm{y}^2 \geq\nb \delta^2$.

	If $b$ \emph{does} lie in~$K$, then $b =\nb Ax$ for some~$x$ in~$\R_+^n$. 
	It follows that for any~$y$ in~$\R^m$, $\inner{b,y} = \sum_{i=1}^n x^i \, 
	\inner{a_i,y}$. But then $\inner{b,y} <\nb 0$ implies $\inner{a_i,y} <\nb 
	0$ for some~$a_i$.
\end{proof}

\begin{remark}
	The point~$v$ is actually the unique point in~$K$ closest to~$b$. To see 
	this, recall that $y = v-\nb b$, let $w$ be in~$K$, and write $z = w-\nb 
	b$. Then the point $\tfrac12 \, (v+\nb w)$ lies in~$K$, so $\norm{y+\nb z} 
	= \norm{v+\nb w-\nb 2b} \geq\nb 2\delta$. Hence, $v\neq\nb w$ implies
	\[
		2\,\norm{z}^2
		=    \norm{y-z}^2 + \norm{y+z}^2 - 2\,\norm{y}^2
		\geq \norm{v-w}^2 + 2\delta^2
		>    2\delta^2.
	\]
\end{remark}

\begin{remark}
	Let $c$ be the infimum of~$\norm{Au}$ over all optimal vectors~$u$, and 
	choose a sequence~$\{u_k\}$ of optimal vectors such that $\norm{Au_k} 
	\to\nb c$. Then the argument from the proof of our lemma yields $c >\nb 
	0$, which is a key step in Bonnans and Shapiro.
\end{remark}

\paragraph{\bfseries Acknowledgments.}
The author would like to thank Gerd Wachsmuth for calling his attention to 
reference~\cite{BS2000}, and three anonymous reviewers for their constructive 
suggestions for improving the presentation.

\end{document}